\documentclass[reqno, 12pt]{amsart}
\usepackage{amssymb}
\usepackage{amsfonts}
\usepackage{srcltx}

\usepackage{cite}

\newtheorem{theorem}{Theorem}[section]
\newtheorem{lemma}[theorem]{Lemma}

\newtheorem{corollary}[theorem]{Corollary}
\theoremstyle{definition}

\theoremstyle{remark}

\newtheorem*{acknowledgement}{Acknowledgement}
\DeclareMathOperator{\diam}{diam}

\DeclareMathOperator{\conv}{conv}

\DeclareMathOperator{\WCS}{WCS}
\DeclareMathOperator{\Nor}{N}

\DeclareMathOperator{\cl}{cl}

\begin{document}
\title[]{On the fixed points of nonexpansive mappings\\
in direct sums of Banach spaces}
\author[]{Andrzej Wi\'{s}nicki}
\address{Andrzej Wi\'{s}nicki, Institute of Mathematics, Maria Curie-Sk{\l }%
odowska University, 20-031 Lublin, Poland}
\email{awisnic@hektor.umcs.lublin.pl}

\begin{abstract}
We show that if a Banach space $X$ has the weak fixed point property for
nonexpansive mappings and $Y$ has the generalized Gossez-Lami Dozo property
or is uniformly convex in every direction, then the direct sum $X\oplus Y$
with a strictly monotone norm has the weak fixed point property.
The result is new even if $Y$ is finite-dimensional.
\end{abstract}

\date{}
\maketitle

\section{Introduction}

One of the central themes in metric fixed point theory is the existence of
fixed points of nonexpansive mappings. Recall that a mapping $T:C\rightarrow
C$ is nonexpansive if
\begin{equation*}
\Vert Tx-Ty\Vert \leq \Vert x-y\Vert
\end{equation*}%
for all $x,y\in C$. A Banach space $X$ is said to have the fixed point
property (FPP) if every nonexpansive self-mapping defined on a nonempty
bounded closed and convex set $C\subset X$ has a fixed point. A Banach space
$X$ is said to have the weak fixed point property (WFPP) if every
nonexpansive self-mapping defined on a nonempty weakly compact and convex
set $C\subset X$ has a fixed point.

Fixed point theory for nonexpansive mappings has its origins in the 1965
existence theorems of F. Browder, D. G\"{o}hde and W. A. Kirk. The most
general of them, Kirk's theorem \cite{Ki} asserts that all Banach spaces
with weak normal structure have WFPP. Recall that a Banach space $X$ has
weak normal structure if $r(C)<\diam C$ for all weakly compact convex
subsets $C$ of $X$ consisting of more than one point, where $r(C)=\inf_{x\in
C}\sup_{x\in C}\Vert x-y\Vert $ is the Chebyshev radius of $C$. In 1981,
Alspach \cite{Al} showed an example of a nonexpansive mapping defined on a
weakly compact convex subset of $L_{1}[0,1]$ without a fixed point, and
Maurey \cite{Ma} used the Banach space ultraproduct construction to prove
FPP for all reflexive subspaces of $L_{1}[0,1]$ as well as WFPP for $c_{0}$
and $H^{1}$. Maurey's method has been applied by numerous authors to obtain
several fixed point results. In 2003, Garc\'{\i}a Falset, Llor\'{e}ns Fuster
and Mazcu\~{n}an Navarro \cite{GLM2} solved a long-standing problem in the
theory by proving FPP for all uniformly nonsquare Banach spaces. Quite
recently, Lin \cite{Li2} showed the first example of a nonreflexive Banach
space with FPP, and Dom\'{\i}nguez Benavides \cite{Do3} proved
that every reflexive Banach space can be renormed to have FPP, thus
solving other classical problems in metric fixed point theory. It is still
unknown whether reflexivity (or even superreflexivity) implies the fixed
point property. For a detailed exposition of metric fixed point theory we
refer the reader to \cite{AyDoLo, GoKi, KiSi}.

The problem of whether FPP or WFPP is preserved under direct sums of Banach
spaces has been studied since the 1968 Belluce--Kirk--Steiner theorem \cite%
{BKS}, which states that a direct sum of two Banach spaces with normal
structure, endowed with the maximum norm, also has normal structure. In
1984, Landes \cite{La1} showed that normal structure is preserved under a
large class of direct sums including all $\ell _{p}^{N}$-sums, $1<p\leq
\infty $, but not under $\ell _{1}^{N}$-direct sums (see \cite{La2}).
Nowadays, there are many results concerning permanence properties of
conditions which imply normal structure, see \cite{DhSa, MaPiXu, SiSm} and
references therein. Several recent papers consider the general case, but
always under additional geometrical assumptions, see \cite{BePr0,BePr,DKK,
DoLiTu,DoSa,KaSaTa,KaTa,PrWi,Wi3}.

Recently, two general fixed point theorems in direct sums were proved in
\cite{PrWi}. In the present paper we are able to remove additional
assumptions imposed on the space $X$ in that paper. We show in Section 3
that if a Banach space $X$ has WFPP and $Y$ has the generalized Gossez-Lami
Dozo property introduced in \cite{Ji} (see Section 2 for the definition),
then the direct sum $X\oplus Y$ with respect to a strictly monotone norm has
WFPP. The result is new even if $Y$ is a finite-dimensional space and in
this case answers a question of Khamsi \cite{Kh} for strictly monotone
norms. Some consequences of the main theorem are presented in Section 4. In
particular, we prove that $X\oplus Y$ has WFPP whenever $X$ has WFPP and $Y$
is uniformly convex in every direction.

\section{Preliminaries}

Let us recall several properties of a Banach space $X$ which are sufficient
for weak normal structure. The normal structure coefficient is given by
\begin{equation*}
\Nor(X)=\inf \left\{ \diam A/r(A)\right\} ,
\end{equation*}%
where the infimum is taken over all bounded convex sets $A\subset X$ with $%
\diam A>0$ and $r(A)$ denotes the Chebyshev radius of $A$ (relative to
itself). Assuming that $X$ does not have the Schur property, we put
\begin{equation*}
\WCS(X)\,=\inf \left\{ \diam_{a}(x_{n})/r_{a}(x_{n})\right\} ,
\end{equation*}%
where the infimum is taken over all sequences $(x_{n})$ which converge to $0$
weakly but not in norm. Here
\begin{equation*}
\diam_{a}(x_{n})=\lim_{n\rightarrow \infty }\sup_{k,l\geq n}\Vert
x_{k}-x_{l}\Vert
\end{equation*}%
denotes the asymptotic diameter of $(x_{n})$ and
\begin{equation*}
r_{a}(x_{n})=\inf \left\{ \limsup_{n\rightarrow \infty }\Vert x_{n}-x\Vert
:x\in \overline{\conv}(x_{n})_{n=1}^{\infty }\right\}
\end{equation*}%
denotes the asymptotic radius of $(x_{n})$. We say that a Banach space $X$
has uniform normal structure if $N(X)>1$ and weak uniform normal structure
(or satisfies Bynum's condition) if $\WCS(X)>1$. A weaker property was
introduced in \cite{Ji}. A Banach space $X$ is said to have the generalized
Gossez-Lami Dozo property (GGLD, in short) if
\begin{equation*}
\limsup_{m\rightarrow \infty }\limsup_{n\rightarrow \infty }\Vert
x_{n}-x_{m}\Vert >1
\end{equation*}%
whenever $(x_{n})$ converges weakly to $0$ and $\lim_{n\rightarrow \infty
}\Vert x_{n}\Vert =1.$ It is known that $\Nor(X)>1$ $%
\Rightarrow $ $\WCS(X)>1$ $\Rightarrow $ GGLD $\Rightarrow $ weak normal
structure and that the GGLD property is equivalent
to the so-called property asymptotic (P) (see, e.g., \cite{SiSm}).

Recall that a norm $\left\| \cdot \right\|$ on $\mathbb{R}^{2}$ is said to
be monotone if
\begin{equation*}
\| ( x_{1},y_{1}) \|\leq \| ( x_{2},y_{2}) \| \ \ \ \text{whenever}\ \ 0\leq
x_{1}\leq x_{2}, 0\leq y_{1}\leq y_{2}.
\end{equation*}
A norm $\left\| \cdot \right\|$ is said to be strictly monotone if
\begin{align*}
\ \| ( x_{1},y_{1}) \|<\| ( x_{2},y_{2}) \| \ \ \ \text{whenever} \ \ &
0\leq x_{1}\leq x_{2}, 0\leq y_{1}<y_{2} \ \  \\
\text{or} \ \ & 0\leq x_{1}<x_{2}, 0\leq y_{1}\leq y_{2}.
\end{align*}
It is easy to see that $\ell _{p}^{2}$-norms, $1\leq p<\infty ,$ are
strictly monotone.

Let $Z$ be a normed space $(\mathbb{R}^{2},\left\| \cdot \right\|_Z)$. We
shall write $X\oplus_{Z}Y$ for the $Z$-direct sum of Banach spaces $X$, $Y$
with the norm
\begin{equation*}
\| ( x,y) \| =\| ( \| x\| ,\| y\|) \|_Z \, ,
\end{equation*}
where $( x,y) \in X\times Y$. The following lemma was proved in \cite[Lemma 4%
]{PrWi}. Similar arguments can be found in \cite{GaLl, SiSm}.

\begin{lemma}
\label{y=0} Let $X\oplus_{Z}Y$ be a direct sum of Banach spaces $X$, $Y$
with respect to a strictly monotone norm. Assume that $Y$ has the GGLD
property, the vectors $w_{n}=(x_{n},y_{n})\in X\oplus _{p}Y$ tend weakly to
0 and
\begin{equation*}
\lim_{n,m\rightarrow \infty ,n\neq m}\Vert w_{n}-w_{m}\Vert
=\lim_{n\rightarrow \infty }\Vert w_{n}\Vert .
\end{equation*}%
Then $\lim_{n\rightarrow \infty }\Vert y_{n}\Vert =0$.
\end{lemma}

\section{The Main Theorem}

The following observation is crucial for many fixed point existence theorems
for nonexpansive mappings. Assume that there exists a nonexpansive mapping $%
T:C\rightarrow C$ without a fixed point, where $C$ is a nonempty weakly
compact convex subset of a Banach space $X$. Let
\begin{equation*}
\mathcal{F}=\left\{ K\subset C:K\;is\;nonempty,\;closed,\;convex\;and\;\
T(K)\subset K\right\} .
\end{equation*}%
From the weak compactness of $C$, any decreasing chain of elements in $%
\mathcal{F}$ has a nonempty intersection which belongs to $\mathcal{F}$. By
the Kuratowski--Zorn lemma, there exists a minimal (in the sense of
inclusion) convex and weakly compact set $K\subset C$ which is invariant
under $T$ and which is not a singleton. Let $(x_{n})$ be an approximate
fixed point sequence for $T$ in $K$, i.e., $\lim_{n\rightarrow \infty }\Vert
Tx_{n}-x_{n}\Vert =0.$ It was proved independently by Goebel \cite{Go} and
Karlovitz \cite{Ka} that
\begin{equation*}
\lim_{n\rightarrow \infty }\Vert x_{n}-x\Vert =\diam K
\end{equation*}%
for every $x\in K$. A fruitful approach to the fixed point problem is to use
this special feature of minimal invariant sets.

Let $T:K\rightarrow K$ be a nonexpansive mapping, where $K$ is a weakly
compact convex subset of the direct sum $X\oplus _{Z}Y$ with respect to a
strictly monotone norm, which is minimal invariant for $T$.

Under suitable conditions imposed on the Banach space $Y$, we will show that
$K$ is isometric to a subset of $X$, thus proving, that $X\oplus _{Z}Y$ has
WFPP whenever $X$ does. To this end we first construct, for every
integer $k\geq 1$, an appropriate family of subsets of $K$ as follows.

\begin{lemma}
\label{sets} Assume that $T:K\rightarrow K$ is a nonexpansive mapping
defined on a weakly compact convex subset $K$ of $X\oplus _{Z}Y$, which is
minimal invariant for $T$ and $\diam K=1$. Let $(w_{n})=((w_{n}^{\prime
},w_{n}^{\prime \prime }))$ be an approximate fixed point sequence for $T$
in $K$ weakly converging to $(0,0)\in K$ and $\lim_{n\rightarrow \infty
}\Vert w_{n}^{\prime \prime }\Vert =0.$ Fix an integer $k\geq 1$ and a
sequence $(\varepsilon _{n})$ in $(0,1).$ Then there exist a subsequence $%
(v_{n})=(x_{n},y_{n})$ of $(w_{n})$ and a family $\left\{ D_{j}^{i}\right\}
_{1\leq j\leq k,i\geq 1}$ of relatively compact convex subsets of $K$ such
that

\begin{enumerate}
\item[(i)] $\Vert Tv_{i}-v_{i}\Vert <\varepsilon _{i},$

\item[(ii)] $\Vert y_{i}\Vert <\varepsilon _{i},$

\item[(iii)] $\Vert v_{i}-z\Vert >1-\varepsilon _{i}$ for all $z\in
D_{k}^{i-1},$

\item[(iv)] $D_{1}^{i}=\conv(D_{1}^{i-1}\cup \left\{ v_{i}\right\} ),$

\item[(v)] $D_{j+1}^{i}=\conv(D_{j}^{i}\cup T(D_{j}^{i})),$
\end{enumerate}

for every $i\geq 1$ and $1\leq j\leq k-1$ $(D_{1}^{0}=D_{k}^{0}=\emptyset ).$
\end{lemma}

\begin{proof}
We proceed by induction on i. Since $\Vert Tw_{n}-w_{n}\Vert $ and $\Vert
w_{n}^{\prime \prime }\Vert $ converge to $0,$ we can choose $%
v_{1}=w_{n_{1}}=(x_{1},y_{1})$ in such a way that $\Vert Tv_{1}-v_{1}\Vert
<\varepsilon _{1}$ and $\Vert y_{1}\Vert <\varepsilon _{1}.$ Let us put
\begin{equation*}
D_{1}^{1}=\left\{ v_{1}\right\}
\end{equation*}%
and, for a given relatively compact convex set $D_{j}^{1},1\leq j<k$,
\begin{equation*}
D_{j+1}^{1}=\conv(D_{j}^{1}\cup T(D_{j}^{1})).
\end{equation*}%
By induction on $j$, we obtain a family $\left\{
D_{1}^{1},...,D_{k}^{1}\right\} $ of relatively compact convex subsets of $K$
which satisfies the desired conditions.

Now suppose that we have chosen $n_{1}<...<n_{l}$ $(l\geq 1),$ $%
v_{i}=w_{n_{i}}=(x_{i},y_{i}),1\leq i\leq l,$ and a family $\left\{
D_{j}^{i}\right\} _{1\leq j\leq k,1\leq i\leq l}$ of relatively compact
convex subsets of $K$ such that the conditions (i)-(v) are satisfied for
every $1\leq i\leq l$ and $1\leq j\leq k-1.$ Then, there exist $%
n_{l+1}>n_{l} $, $v_{l+1}=w_{n_{l+1}}=(x_{l+1},y_{l+1})$ such that $\Vert
Tv_{l+1}-v_{l+1}\Vert <\varepsilon _{l+1}$, $\Vert y_{l+1}\Vert <\varepsilon
_{l+1}$ and $\Vert v_{l+1}-z\Vert >1-\varepsilon _{l+1}$ for all $z\in
D_{k}^{l}$ (the last inequality follows from the Goebel-Karlovitz lemma and
the relative compactness of $D_{k}^{l}$). Let us put
\begin{equation*}
D_{1}^{l+1}=\conv(D_{1}^{l}\cup \left\{ v_{l+1}\right\} )
\end{equation*}%
and, for a given relatively compact convex set $D_{j}^{l+1},1\leq j<k$,
\begin{equation*}
D_{j+1}^{l+1}=\conv(D_{j}^{l+1}\cup T(D_{j}^{l+1})).
\end{equation*}%
Then, by induction with respect to $j$, we obtain a family $\left\{
D_{1}^{l+1},...,D_{k}^{l+1}\right\} $ of relatively compact convex subsets
of $K$ which satisfies the desired conditions.

By induction on $i$, the lemma follows.
\end{proof}

We are now going to prove that for a sequence $(\varepsilon _{n}(k))$, if $%
u=(a,b)\in \bigcup\nolimits_{i\mathbb{=}1}^{\infty }D_{k}^{i}(k)$ and $k$ is
large, then $b$ is close to $0$. We need the following lemma.

\begin{lemma}
\label{close1} Assume that a sequence $(v_{n})=(x_{n},y_{n})$ and a family $%
\left\{ D_{j}^{i}\right\} _{1\leq j\leq k,i\geq 1}$ of relatively compact
convex subsets of $K$ are given as in Lemma \ref{sets}. Then, for every $%
1\leq j\leq k$, $i\geq 1$ and $u\in D_{j}^{i+1}$, there exists $z\in
D_{j}^{i}$ such that
\begin{equation*}
\left\Vert z-u\right\Vert +\left\Vert u-v_{i+1}\right\Vert \leq \left\Vert
z-v_{i+1}\right\Vert +3(j-1)\varepsilon _{i+1}.
\end{equation*}
\end{lemma}

\begin{proof}

Fix $i\geq 1.$ We proceed by induction with respect to $j$. For $j=1$ and $%
u\in D_{1}^{i+1}=\conv(D_{1}^{i}\cup \left\{ v_{i+1}\right\} )$ there exists
$z\in D_{1}^{i}$ such that
\begin{equation*}
\left\Vert z-u\right\Vert +\left\Vert u-v_{i+1}\right\Vert =\left\Vert
z-v_{i+1}\right\Vert .
\end{equation*}

Now fix $1\leq j<k$ and suppose that for every $u\in D_{j}^{i+1}$ there
exists $z\in D_{j}^{i}$ such that
\begin{equation}
\left\Vert z-u\right\Vert +\left\Vert u-v_{i+1}\right\Vert \leq \left\Vert
z-v_{i+1}\right\Vert +3(j-1)\varepsilon _{i+1}.  \label{step3_U}
\end{equation}%
Let
\begin{equation*}
u\in D_{j+1}^{i+1}=\conv(D_{j}^{i+1}\cup T(D_{j}^{i+1})).
\end{equation*}%
Consider three cases.

$1^{\circ}$ The inductive step is obvious if $u\in D_{j}^{i+1}.$

$2^{\circ }$ Let $u\in T(D_{j}^{i+1})$. Then $u=T\bar{u}$ for some $\bar{u}%
\in D_{j}^{i+1}$ and, by assumption, there exists $\bar{z}\in D_{j}^{i}$
such that
\begin{equation*}
\left\Vert \bar{z}-\bar{u}\right\Vert +\left\Vert \bar{u}-v_{i+1}\right\Vert
\leq \left\Vert \bar{z}-v_{i+1}\right\Vert +3(j-1)\varepsilon _{i+1}.
\end{equation*}%
Let $z=T\bar{z}\in D_{j+1}^{i}\subset D_{k}^{i}$. Then
\begin{align}
& \left\Vert z-u\right\Vert +\left\Vert u-v_{i+1}\right\Vert \leq \left\Vert
T\bar{z}-T\bar{u}\right\Vert +\left\Vert T\bar{u}-Tv_{i+1}\right\Vert
+\left\Vert Tv_{i+1}-v_{i+1}\right\Vert   \notag \\
& <\left\Vert \bar{z}-\bar{u}\right\Vert +\left\Vert \bar{u}%
-v_{i+1}\right\Vert +\varepsilon _{i+1}\leq \left\Vert \bar{z}%
-v_{i+1}\right\Vert +(3j-2)\varepsilon _{i+1}  \label{step3b_U} \\
& <\left\Vert z-v_{i+1}\right\Vert +(3j-1)\varepsilon _{i+1},  \notag
\end{align}%
since, by (i), $\left\Vert Tv_{i+1}-v_{i+1}\right\Vert <\varepsilon _{i+1}$
and, by (iii), $\left\Vert z-v_{i+1}\right\Vert >1-\varepsilon _{i+1}\geq
\left\Vert \bar{z}-v_{i+1}\right\Vert -\varepsilon _{i+1}$ $(\diam K=1).$

$3^{\circ }$ Let $u=\sum\nolimits_{s=1}^{t}\lambda _{s}u_{s}$ for some $%
u_{s}\in D_{j}^{i+1}\cup T(D_{j}^{i+1}),\lambda _{s}\in \left[ 0,1\right]
,1\leq s\leq t\in \mathbb{N}$, $\sum\nolimits_{s=1}^{t}\lambda _{s}=1.$
Then, by (\ref{step3_U}) or (\ref{step3b_U}), there exist $%
z_{1},...,z_{t}\in D_{j+1}^{i}$ such that%
\begin{equation*}
\left\Vert z_{s}-u_{s}\right\Vert +\left\Vert u_{s}-v_{i+1}\right\Vert \leq
\left\Vert z_{s}-v_{i+1}\right\Vert +(3j-1)\varepsilon _{i+1},1\leq s\leq t.
\end{equation*}%
Hence%
\begin{align*}
& \Vert \sum_{s=1}^{t}\lambda _{s}z_{s}-u\Vert +\left\Vert
u-v_{i+1}\right\Vert \leq \sum_{s=1}^{t}\lambda _{s}\left\Vert
z_{s}-v_{i+1}\right\Vert +(3j-1)\varepsilon _{i+1} \\
& \leq 1+(3j-1)\varepsilon _{i+1}<\Vert \sum_{s=1}^{t}\lambda
_{s}z_{s}-v_{i+1}\Vert +3j\varepsilon _{i+1},
\end{align*}%
since, by (iii), $\Vert \sum_{s=1}^{t}\lambda _{s}z_{s}-v_{i+1}\Vert
>1-\varepsilon _{i+1}.$

By induction on $j$, the lemma follows.
\end{proof}

\begin{lemma}
\label{close2} Let $K$ be a subset of a direct sum $X\oplus _{Z}Y$ endowed
with a strictly monotone norm. Under the assumptions of Lemma \ref{sets},
for every positive integer $k$, there exist a sequence $(\varepsilon
_{n}(k)) $ in $(0,1),$ a subsequence $(v_{n}(k))=(x_{n}(k),y_{n}(k))$ of $%
(w_{n})$ and a family $\left\{ D_{j}^{i}(k)\right\} _{1\leq j\leq k,i\geq 1}$
of relatively compact convex subsets of $K$ such that $\left\Vert
b\right\Vert <\frac{1}{k}$ for every $u=(a,b)\in \bigcup\nolimits_{i\mathbb{=%
}1}^{\infty }D_{k}^{i}(k).$
\end{lemma}

\begin{proof}
Since $Z=(\mathbb{R}^{2},\left\Vert \cdot \right\Vert _{Z})$ is a finite
dimensional space and the norm $\left\Vert \cdot \right\Vert _{Z}$ is strictly
monotone, for every $%
\varepsilon >0,$ there exists $\delta (\varepsilon )>0$ such that if $(\bar{a%
},\bar{b}),(\bar{a},\bar{c})$ belong to the unit ball $B_{Z}$ and $%
\left\Vert (\bar{a},\bar{b})\right\Vert <\left\Vert (\bar{a},\bar{c}%
)\right\Vert +\delta (\varepsilon ),$ then $\left\Vert \bar{b}\right\Vert
<\left\Vert \bar{c}\right\Vert +\varepsilon .$
Fix $k\geq 1,$ $\eta =\frac{1}{4k}$ and choose
\begin{equation*}
\varepsilon _{i}=\varepsilon _{i}(k)<\min \left\{ \frac{\delta (\eta ^{i})}{%
3k},\frac{\eta ^{i}}{k}\right\} ,\ \ i\geq 1.
\end{equation*}%
By Lemma \ref{sets}, there exist a sequence $(v_{n}(k))=(x_{n}(k),y_{n}(k))$
and a family $\left\{ D_{j}^{i}(k)\right\} _{1\leq j\leq k,i\geq 1}$ of
relatively compact convex subsets of $K$ with the properties described in
this lemma.

Let $u=(a,b)\in D_{k}^{i}(k),i\geq 2.$ It follows from Lemma \ref{close1}
that there exists $z=(x,y)\in D_{k}^{i-1}(k)$ such that
\begin{equation*}
\left\Vert z-u\right\Vert +\left\Vert u-v_{i}(k)\right\Vert \leq \left\Vert
z-v_{i}(k)\right\Vert +3(k-1)\varepsilon _{i}<\left\Vert
z-v_{i}(k)\right\Vert +3k\varepsilon _{i}.
\end{equation*}%
Hence
\begin{equation*}
\left\Vert (\left\Vert x-x_{i}(k)\right\Vert ,\left\Vert y-b\right\Vert
+\left\Vert b-y_{i}(k)\right\Vert )\right\Vert <\left\Vert (\left\Vert
x-x_{i}(k)\right\Vert ,\left\Vert y-y_{i}(k)\right\Vert )\right\Vert
+3k\varepsilon _{i}
\end{equation*}%
which yields
\begin{equation*}
\left\Vert y-b\right\Vert +\left\Vert b-y_{i}(k)\right\Vert <\left\Vert
y-y_{i}(k)\right\Vert +\eta ^{i}.
\end{equation*}%
Consequently,
\begin{equation*}
\left\Vert b\right\Vert <\left\Vert y\right\Vert +\left\Vert
y_{i}(k)\right\Vert +\frac{1}{2}\eta ^{i}.
\end{equation*}%
By induction with respect to $i$, there exists $(\bar{x},\bar{y})\in
D_{k}^{1}(k)$ such that
\begin{equation*}
\left\Vert b\right\Vert <\left\Vert \bar{y}\right\Vert +(\varepsilon
+...+\varepsilon _{i})+\frac{1}{2}(\eta +...+\eta ^{i})<k\varepsilon
_{1}+2\eta +\eta <4\eta =\frac{1}{k}.
\end{equation*}
\end{proof}

We are now in a position to prove the main theorem.

\begin{theorem}
\label{Th1}Let $X$ be a Banach space with WFPP and $Y$ has the GGLD
property. Then $X\oplus _{Z}Y$ with respect to a strictly monotone norm has
WFPP.
\end{theorem}

\begin{proof}
Assume that $X\oplus _{Z}Y$ does not have WFPP. Then, there exist a weakly
compact convex subset $C$ of $X\oplus _{Z}Y$ and a nonexpansive mapping $%
T:C\rightarrow C$ without a fixed point. By the Kuratowski-Zorn lemma, there
exists a convex and weakly compact set $K\subset C$ which is minimal
invariant under $T$ and which is not a singleton. Let $(w_{n})=((w_{n}^{%
\prime },w_{n}^{\prime \prime }))$ be an approximate fixed point sequence
for $T$ in $K$, i.e., $\lim_{n\rightarrow \infty }\Vert Tw_{n}-w_{n}\Vert =0$%
. Without loss of generality we can assume that $\diam K=1$, $(w_{n})$
converges weakly to $(0,0)\in K$ and the double limit $\lim_{n,m\rightarrow
\infty ,n\neq m}\Vert w_{n}-w_{m}\Vert $ exists. It follows from the
Goebel-Karlovitz lemma that
\begin{equation}
\lim_{n,m\rightarrow \infty ,n\neq m}\Vert w_{n}-w_{m}\Vert
=\lim_{n\rightarrow \infty }\Vert w_{n}\Vert =1.
\end{equation}%
Applying Lemma \ref{y=0} gives $\lim_{n\rightarrow \infty }\Vert
w_{n}^{\prime \prime }\Vert =0.$ Lemma \ref{close2} now shows that for every
positive integer $k,$ there exist a subsequence $%
(v_{n}(k))=(x_{n}(k),y_{n}(k))$ of $(w_{n})$ and a family $\left\{
D_{j}^{i}(k)\right\} _{1\leq j\leq k,i\geq 1}$ of relatively compact convex
subsets of $K$ such that $\left\Vert b\right\Vert <\frac{1}{k}$ for every $%
u=(a,b)\in \bigcup\nolimits_{i\mathbb{=}1}^{\infty }D_{k}^{i}(k).$

Let $C_{0}=\{(0,0)\}$ and $C_{j}=\conv(C_{j-1}\cup T(C_{j-1}))$ for $j\geq 1.
$ It is not difficult to see that $\cl(\bigcup\nolimits_{j\mathbb{=}%
1}^{\infty }C_{j})$ is a closed convex subset of $K$ which is invariant for $%
T$ (and hence equals $K$). Fix $k\geq 1$ and notice that $(0,0)\in \cl(
\bigcup\nolimits_{i\mathbb{=}1}^{\infty }D_{1}^{i}(k))$, because a sequence $%
(v_{n}(k))_{n\geq 1}$ converges weakly to $(0,0).$ Furthermore,
\begin{equation*}
T(\cl(\bigcup\nolimits_{i\mathbb{=}1}^{\infty }D_{j}^{i}(k)))=\cl%
(\bigcup\nolimits_{i\mathbb{=}1}^{\infty }T(D_{j}^{i}(k)))\subset \cl%
(\bigcup\nolimits_{i\mathbb{=}1}^{\infty }D_{j+1}^{i}(k))
\end{equation*}
and hence, by induction on $j,$
\begin{equation*}
C_{j}\subset \cl(\bigcup\nolimits_{i\mathbb{=}1}^{\infty
}D_{j+1}^{i}(k))\subset \cl(\bigcup\nolimits_{i\mathbb{=}1}^{\infty
}D_{k}^{i}(k)),\ j<k.
\end{equation*}
It follows that if $(x,y)\in C_{j}$ and $j<k,$ then $\left\Vert y\right\Vert
\leq \frac{1}{k}.$ Since $k$ is arbitrary, $y=0$ for every $(x,y)\in \cl%
(\bigcup\nolimits_{j\mathbb{=}1}^{\infty }C_{j})=K.$ Therefore, $K$ is
isometric to a subset of $X$. Since $X$ has WFPP, $T$ has a fixed point in $%
K,$ which contradicts our assumption.
\end{proof}

\section{Consequences}

In this section, we list some consequences of Theorem \ref{Th1}. Notice that
in the case of reflexive spaces, the properties FPP and WFPP coincide.
Furthermore, if a Banach space $Y$ has uniform normal structure ($\Nor(Y)>1$%
), then $Y$ is reflexive and has FPP. In the remainder of this section, $%
X\oplus _{Z}Y$ denotes a direct sum of Banach spaces $X$ and $Y$ with
respect to a strictly monotone norm.

\begin{corollary}
Suppose $X$ is a reflexive Banach space with FPP and $Y$ has uniform normal
structure. Then $X\oplus _{Z}Y$ has FPP.
\end{corollary}

In particular, the above corollary is valid if $X$ is a uniformly nonsquare
or a uniformly noncreasy Banach space.

\begin{corollary}
Suppose $X$ is a Banach space with WFPP and $Y$ satisfies Bynum's condition $%
\WCS(Y)>1$. Then $X\oplus _{Z}Y$ has WFPP.
\end{corollary}

It is well known that all finite dimensional spaces have uniform normal
structure. A very particular case of the above corollary answers a question
of M. A. Khamsi (see \cite[p. 999]{Kh}) for strictly monotone norms.

\begin{corollary}
\label{3}Suppose $X$ is a Banach space with WFPP and $Y$ is a finite
dimensional space. Then $X\oplus _{Z}Y$ has WFPP.
\end{corollary}

A Banach space $X$ with the property that $X\oplus _{1}\mathbb{R}$ has WFPP
has been studied in \cite{KRS}. The following theorem was established for
the $\ell _{1}^{2}$-norm but the proof is valid for all strictly monotone norms.

\begin{theorem}
[see {\cite[Theorem 1]{KRS}}] {\label{Ku}} Suppose $X$ is a Banach space such that $%
X\oplus _{Z}\mathbb{R}$ has WFPP. Let $Y$ be a Banach space which is
uniformly convex in every direction. Then $X\oplus _{Z}Y$ has WFPP.
\end{theorem}

Corollary \ref{3} and Theorem \ref{Ku} give the following result.

\begin{theorem}
\label{Th2}Suppose $X$ is a Banach space with WFPP and $Y$ is uniformly
convex in every direction. Then $X\oplus _{Z}Y$ has WFPP.
\end{theorem}

Recall that the James space $J$ is an example of a Banach space with the
GGLD property which is not uniformly convex in every direction and the space
$c_{0}$ with the norm%
\begin{equation*}
\left\Vert x\right\Vert =\sqrt{\left\Vert x\right\Vert _{\infty
}^{2}+\sum_{i=1}^{\infty }\frac{x_{i}^{2}}{2^{i}}}
\end{equation*}%
is an example of a Banach space which is uniformly convex in every direction
but fails the GGLD property (see \cite{GaLl} and references therein). This
shows that Theorem \ref{Th1} and Theorem \ref{Th2} are independent of each
other.

\begin{acknowledgement}
The author is greatly indebted to the referee for his valuable advice which
led to a substantial simplification of the original arguments and to a more
general result.
\end{acknowledgement}

\end{document}